\documentclass[a4paper,11pt]{article}
\usepackage{amsmath,amssymb,amsthm,graphicx}
\usepackage{fullpage}
\usepackage{enumitem}

\newtheorem{theorem}{Theorem}[section]
\newtheorem{lemma}[theorem]{Lemma}

\newtheorem{proposition}[theorem]{Proposition}

\theoremstyle{definition}
\newtheorem{example}[theorem]{Example}
\newtheorem{definition}[theorem]{Definition}
\newtheorem{remark}[theorem]{Remark}

\title{Characterising bimodal collections of sets in finite groups}
\author{S. Huczynska and M.B.~Paterson}
\begin{document}
\maketitle
\begin{abstract}
A collection of disjoint subsets ${\cal A}=\{A_1,A_2,\dotsc,A_m\}$ of a finite abelian group is said to have the \emph{bimodal} property if, for any non-zero group element $\delta$, either $\delta$ never occurs as a difference between an element of $A_i$ and an element of some other set $A_j$, or else for every element $a_i$ in $A_i$ there is an element $a_j\in A_j$ for some $j\neq i$ such that $a_i-a_j=\delta$.  This property arises in various familiar situations, such as the cosets of a fixed subgroup or in a group partition, and has applications to the construction of optimal algebraic manipulation detection (AMD) codes.  In this paper,  we obtain a structural characterisation for bimodal collections of sets.
\end{abstract}

\section{Introduction}\label{sec:background}
Let $G$ be a finite abelian group of order $n$, written additively, with identity $0$. Let ${\cal A}=\{A_1,A_2,\dotsc,A_m\}$ be a collection of disjoint subsets of $G$.  Then $\cal A$ is said to have the {\em bimodal property} if, for any non-identity element $\delta$ of $G$, either $\delta$ \emph{never} occurs as a difference between an element of $A_i$ and an element of some other set $A_j$, or else for \emph{every} element $a_i$ in $A_i$ there is an element $a_j\in A_j$ for some $j\neq i$ such that $a_i-a_j=\delta$.

Let $k_i=|A_i|$ for $i=1,2,\dotsc,m$.  For each $\delta\in G\setminus\{0\}$, we let
$$N_i(\delta)=|\{(a_i,a_j): \delta=a_i-a_j,\ a_i\in A_i,\ a_j\in A_j,\ i\neq j\}|.$$  
Then $\cal A$ has the {\em bimodal property} if $N_i(\delta)\in \{0,k_i\}$ for $i=1,2,\dotsc,m$.  

Although this is a very natural property which occurs in many familiar settings, it seems that the property has not received attention until now.  It has applications to cryptography, having been defined by Huczynska and Paterson \cite{WEDF} in the context of studying {\em reciprocally-weighted external difference families} (RWEDFs), which were shown to exhibit this property in certain parameter situations.   RWEDFs correspond to optimal Algebraic Manipulation Detection (AMD) codes \cite{cramer,PatSti}, and better understanding of this property can lead to new constructions for such codes.

Perhaps the most natural occurrence of this property is as follows:
\begin{lemma}\label{lem:cosets}
Let $H$ be a subgroup of an abelian group $G$.  If ${\cal C}=\{C_1, \ldots, C_m \}$ is a collection of cosets of $H$, then ${\cal C}$ has the bimodal property.  
\end{lemma}
\begin{proof}
For fixed $i$ and $1 \leq j \leq m$ with $i \neq j$, the sets $C_i-C_j$ comprise $m-1$ distinct cosets of $H$.  For any $\delta \in C_i-C_j$ and every $x \in C_i$ there exists a unique $y \in C_j$ such that $x-y= \delta$.  For any $\delta \in G \setminus \cup_{j \neq i} (C_i-C_j)$, the element $\delta$ occurs zero times as a difference out of $C_i$.
\end{proof}

\begin{example}
Let $G=\mathbb{Z}_{10}$, $H=\{0,5\}$ and let $C_1=\{1,6\}$, $C_2=\{3,8\}$ and $C_3=\{4,9\}$.  We observe that $N_3(1)=N_3(3)=N_3(6)=N_3(8)=2=k_3$, whereas $N_3(2)=N_3(4)=N_3(5)=N_3(7)=N_3(9)=0$.  A similar check of the values of $N_1(\delta)$ and $N_2(\delta)$ for $\delta\in \mathbb{Z}_{10}\setminus\{0\}$ shows that ${\cal A}=\{A_1,A_2,A_3\}$ has the bimodal property.
\end{example}

This is such a natural setting in which the bimodal property occurs, that one might initially suspect that cosets of a fixed subgroup are the only nontrivial collections of sets with this property. However, we shall show that this is not in fact the case, and that a much richer structure arises.

In Section~\ref{sec:basicresults} of this paper we make some fundamental observations about properties of bimodal collections of sets, and give examples and techniques for creating new from old.  In Section~\ref{sec:classification} we classify bimodal collections of sets by considering certain subgroups associated to them, and we demonstrate that every such collection can be constructed in a suitable group.  

\section{Tools and Constructions}\label{sec:basicresults}
The bimodal property is defined in terms of differences between elements lying in distinct members of a collection ${\cal A}=\{A_1,A_2,\dotsc,A_m\}$ of disjoint subsets of a group $G$ ({\em external} differences).  

The following notation will be used throughout the paper.  Let $A=\bigcup_{i=1}^m A_i$, and for $j$ with $1\leq j \leq m$ let $B_j=A\setminus A_j$.  So for each $1 \leq j \leq m$, $A= A_j \sqcup B_j$ (where the notation $S\sqcup T$ denotes the union of disjoint sets $S$ and $T$).  Then 
$$N_j(\delta)=|\{(a,b):\delta=a-b,\ a\in A_j,\ b\in B_j\}|.$$  
We will also find it useful to consider differences between elements within the same subset ({\em internal} differences).  
\begin{definition}
Let $A_i$ be a subset of a finite abelian group $G$.  We define the {\em internal difference group} $H_i$ of $A_i$ to be the subgroup of $G$ generated by all elements of the form $x-y$ with $x,y\in A_i$.
\end{definition}

In Section~\ref{sec:background} we saw a bimodal collection of sets where each set was a coset of a subgroup $H$ of $G$.  In this case we have $H_i=H$ for each set $A_i$ in the collection.  
\begin{remark}
The group $H_i$ has the property that $A_i$ is contained in a single coset of $H_i$, and is the smallest subgroup of $G$ with this property.  In the case where $|A_i|=1$, the group $H_i=\{0\}$ and its cosets are the singleton sets. 
\end{remark}

We may view singleton sets as cosets of the identity subgroup; throughout what follows we will take this approach.  

For any element $a \in A_i$ we have that $a+H_i$ is the coset of $H_i$ containing $A_i$.  Throughout what follows we will select an arbitrary element $a_i\in A_i$ for $i=1,2,\dotsc,m$ and represent the coset of $H_i$ containing $A_i$ by $a_i+H_i$.

The following useful characterisation of bimodal collections of sets was established in \cite{WEDF}.
\begin{theorem}[\cite{WEDF}]\label{thm:bimodal}
Let $G$ be a finite abelian group and let ${\cal A}=\{A_1,A_2,\dotsc,A_m\}$ be a collection of disjoint subsets of $G$.  Then $\cal A$ has the bimodal property if and only if for each $i$ the set $B_i$ is a union of cosets of the subgroup $H_i$.
\end{theorem}
\begin{remark}\label{remark2.4}
Theorem~\ref{thm:bimodal} tells us that if an element $v$ is contained in $B_j$ for some $j$ with $1\leq j\leq m$ then $B_j$ contains the entire coset $v+H_j$, i.e. we have that $v+h_j\in B_j$ for all $h_j\in H_j$.
\end{remark}

\begin{example}\label{ex:trivialbimodal}
\begin{itemize}
\item[(i)]  A family of disjoint subsets of an abelian group $G$ in which all subsets have size $1$ is bimodal by Theorem \ref{thm:bimodal}.  
\item[(ii)] A family consisting of a single subset of an abelian group $G$ is trivially bimodal.
\end{itemize}
\end{example}

We now exhibit a non-trivial example of a bimodal collection of sets that appeared in \cite{WEDF}, which may at first seem surprising.  It is very different in structure to our coset example, but also represents a classic and much-studied group theory situation. For a group $G$ we will let $G^*$ denote the set of nonzero elements of $G$.

\begin{definition}
Let $G$ be a finite group.  If $G$ has subgroups $S_1,S_2,\dotsc, S_m$ with the property that $S_1^*,S_2^*,\dotsc , S_m^*$ partition $G^*$, then the collection of subgroups $S_1,S_2,\dotsc, S_m$ is called a \emph{group partition} of $G$.  A group partition is called \emph{trivial} if $m=1$. 
\end{definition}
\begin{lemma}[\cite{WEDF}]\label{group_partition}
Let the collection of subgroups $\{ S_1, \ldots, S_m\}$ be a group partition of an abelian group $G$ where $|S_i|>2$ for each $i$.  Then the collection of sets ${\cal C}=\{S_1^*, \ldots, S_m^*\}$ has the bimodal property. 
\end{lemma}
\begin{proof}
For each $i$, the internal difference group of $S_i^*$ is $S_i$.  It follows directly that the union $\bigcup_{j \neq i} S_j^*$ is $G \setminus S_i$, a union of cosets of $S_i$.
\end{proof}
\begin{example} Let $G=\mathbb{Z}_3 \times \mathbb{Z}_3$.  Let $A_1=\{(1,1), (2,2)\}$, $A_2=\{(0,1), (0,2)\}$, $A_3=\{(1,2), (2,1)\}$ and $A_4=\{(1,0), (2,0)\}$.  Observe that for each $A_i$, the subgroup $H_i$ is precisely $A_i \cup \{0\}$.  Then the collection $\cal A$ $=\{ A_1, A_2, A_3, A_4 \}$ is bimodal. 
\end{example}

The topic of group partitions has been studied extensively; see Zappa \cite{Zap} for a comprehensive survey.  In the abelian case, the groups which possess non-trivial group partitions are completely characterised as the elementary abelian $p$-groups.  These can be viewed as partitions of vector spaces over $\mathbb{Z}_p$, and have been widely studied in this context; for example, see Heden \cite{heden}.


In fact the construction of Lemma~\ref{group_partition} may be extended to give examples in more general groups:

\begin{lemma}[\cite{WEDF}]
Let $\cal A$ be a collection $ {\cal A}=\{A_1, \ldots, A_m\}$ of disjoint subsets of an abelian group $G$ that partition $G^*$ and have the property that any $A_i$ with $|A_i|>1$ is of the form $S^*$ for some subgroup $S\leq G$.  Then $\cal A$ is bimodal.
\end{lemma}
\begin{example}
Let $G=\mathbb{Z}_{12}$.  Then $G$ has a subgroup $\{0,4,8\}$ of order $3$ and a subgroup $\{0,3,6,9\}$ of order $4$.  The sets $\{4,8\}$, $\{3,6,9\}$, $\{1\}$, $\{2\}$, $\{5\}$, $\{7\}$, $\{10\}$, $\{11\}$ form a bimodal collection.
\end{example}

Having seen some examples of constructions of bimodal collections of sets, we now introduce some approaches to constructing new collections from old, which we exploit in Section~\ref{sec:classification}.  First we observe that a bimodal collection can undergo a shift without its bimodality being affected.

\begin{lemma}\label{lemma:youcanshift}
Let ${\cal A}=\{A_1,A_2,\dotsc,A_m\}$ be a bimodal collection of disjoint subsets of an abelian group $G$. Then the collection ${\cal A}^\prime$ given by $\{ A_1+g, A_2+g, \dotsc, A_m+g\}$ where $g \in G$ is also bimodal.
\end{lemma}
\begin{proof}
The bimodal property is defined entirely in terms of differences between elements in different sets of the collection.  The value of a difference $a_i-a_j$ does not change if $g$ is added to both $a_i$ and $a_j$, and hence the result follows immediately.
\end{proof}

We may also replace a coset by a partition of that coset into smaller cosets, a process we will refer to as {\em subdivision}:


\begin{theorem}\label{thm:subdivision}
Let ${\cal A}=\{A_1,A_2,\dotsc,A_m\}$ be a bimodal collection of disjoint subsets of an abelian group $G$.  Suppose that $A_i$ is a coset of $H_i$ for some $i$.  Let $A_i^1,A_i^2, \dotsc, A_i^r$ be disjoint subsets partitioning $A_i$, having internal difference groups $H_i^1,H_i^2, \dotsc, H_i^r$ respectively, and which possess the property that for $j=1,2,\dotsc, r$ the set $A_i^j$ is a coset of $H_i^j$.  Then the collection ${\cal A^{\prime}}=\{A_1,A_2,\dotsc,A_{i-1},A_i^1,A_i^2,\dotsc,A_i^r ,A_{i+1},\dotsc,A_m\}$ satisfies the bimodal property.  We shall refer to $\cal A^{\prime}$ as a \emph{subdivision} of ${\cal A}$.
\end{theorem}
\begin{proof}
The union of elements in the collection is not changed by subdivision, so for sets $A_j$ with $j\neq i$ it is still the case that $B_j$ is a union of cosets of $H_j$.  Consider a set $A_i^t$; we will denote $A \setminus A_i^t$ by $B_i^t$.   We note that $H_i^t\leq H_i$, and so any coset of $H_i$ is a union of cosets of $H_i^t$.  Now, $B_i^t=(A_i\setminus A_i^t) \cup B_i$.  We have that $B_i$ is a union of cosets of $H_i^t$ since it is a union of cosets of $H_i$, and $A_i\setminus A_i^t$ is a union of cosets of $H_i^t$ as $A_i$ is a union of cosets of $H_i$ and $A_i^t$ is itself a coset of $H_i^t$.  Thus we deduce that the collection is bimodal as required.
\end{proof}
We observe that subdivision can be applied repeatedly if desired, and the resulting collection will still be bimodal.

\begin{example}
Let $G$ be the elementary abelian group of order $8$.  Write it as $G=\mathbb{Z}_2 \times \mathbb{Z}_2 \times \mathbb{Z}_2, +)$.  Then by Example \ref{ex:trivialbimodal}, ${\cal A}=\{ G\}$ is trivially bimodal.  We can apply subdivision to obtain ${\cal A^{\prime}}$, as follows. 
Define the following subgroups of $G$:
$$ S_1=\{ (0,0,0), (0,0,1)\};$$
$$ S_2= \{ (0,0,0), (1,0,0)\};$$
$$ S_3=\{ (0,0,0), (0,1,1)\};$$
$$ S_4=\{ (0,0,0), (1,1,0) \}.$$
Then $G$ may be partitioned as:
$$A_1=S_1=\{ (0,0,0), (0,0,1)\};$$
$$A_2=(0,1,0)+S_2=\{ (0,1,0), (1,1,0)\};$$
$$A_3=(1,0,0)+S_3=\{ (1,0,0), (1,1,1)\};$$
$$A_4=(0,1,1)+S_4=\{ (0,1,1), (1,0,1)\}.$$  
Take ${\cal A^{\prime}}=\{ A_1, A_2, A_3, A_4\}$.  This is bimodal by Theorem \ref{thm:subdivision}.  To see this directly, observe that the differences out of $A_1$ comprise every element $2$ times except the elements of $S_1$ which occur zero times; and in general the differences out of $A_i$  comprise every element twice except the elements of $S_i$ which occur $0$ times.  This means that, for $\delta \in G^*$, if $\delta$ is one of the four non-zero elements of $\cup_{1 \leq i \leq 4}S_i$ then $N_i(\delta)=0$ for precisely one value of $i \in \{1,2,3,4\}$, whereas if $\delta$ is one of the three elements of $G^*$ not in $\cup_{1 \leq i \leq 4}S_i$, then $N_i(\delta)>0$ for all $i$.
\end{example}

\section{Classification of bimodal collections of sets}\label{sec:classification}
We will now develop a complete characterisation of bimodal collections of sets, through a series of results that classify them according to the relationship between their sets $A_i$ and the corresponding internal difference groups $H_i$. 

\begin{remark}
 For a bimodal collection ${\cal A}=\{A_1,A_2,\dotsc,A_m\}$ of disjoint subsets of an abelian group $G$ we have
 $|A_i|\leq |H_i|$ for each $i$.  Without loss of generality, we suppose that the sets are labelled so that for $i=1,2,\dots,r_{\!_{\cal A}}$ we have $|A_i|<|H_i|$, and for $i=r_{\!_{\cal A}}+1,r_{\!_{\cal A}}+2,\dotsc, m$ we have $|A_i|=|H_i|$, i.e. $A_i$ fills the coset $a_i+H_i$ for $i>r_{\!_{\cal A}}$.
\end{remark}

The following technical lemma will be useful in what follows.

\begin{lemma}\label{lem:excitinglemma}
Let $\cal A$ be a collection $A_1,A_2,\dotsc,A_m$ of disjoint subsets of $G$ that has the bimodal property.  Then $A_k \cap (a_j+H_j) = \emptyset$ for any $k\neq j$.
\end{lemma} 
\begin{proof}
$B_j$ is a union of cosets of $H_j$, which does not include the coset $a_j+H_j$.  For $k\neq j$, we have $A_k\subseteq B_j$, and so $A_k\cap (a_j+H_j)=\emptyset$. 
\end{proof}
In other words, not only is $A_k$ disjoint from $A_j$ when $k\neq j$, it is also disjoint from the entire coset of $H_j$ that contains $A_j$.

\subsection{The case when $r_{\!_{\cal A}} \geq 2$ }
We will consider three cases: the situation when there are at least two sets $A_i$ that do not fill $a_i+H_i$ (i.e. $r_{\!_{\cal A}} \geq 2$), the case when precisely one $A_i$ does not fill $a_i+H_i$ (i.e. $r_{\!_{\cal A}}=1$), and the case when each of the sets in ${\cal A}$ fills the coset of $H_i$ that contains it ($r_{\!_{\cal A}}=0$).

\begin{lemma}\label{lem:notsubgroups}
Let ${\cal A}=\{A_1,A_2,\dotsc,A_m\}$ be a bimodal collection of disjoint subsets of an abelian group $G$ with $r_{\!_{\cal A}}\geq 2$.  Then for $i,j\leq r_{\!_{\cal A}}$ with $i\neq j$ we have that $H_i$ is not a subgroup of $H_j$ and $H_j$ is not a subgroup of $H_i$.
\end{lemma}
\begin{proof}
Let ${\cal A}=\{A_1,A_2,\dotsc,A_m\}$ be a bimodal collection of disjoint subsets of an abelian group $G$ with $r_{\!_{\cal A}}\geq 2$, and suppose that for some $i,j\leq r_{\!_{\cal A}}$ with $i\neq j$ we have $H_i\leq H_j$.   Since $|A_i|<|H_i|$ there are elements of $a_i+H_i\subseteq a_i+H_j$ that are not contained in $B_j$, and by Lemma~\ref{lem:excitinglemma} these elements are not contained in any other set in $\cal A$.  It follows that $B_j$ is not a union of cosets of $H_j$, and so $\cal A$ is not bimodal.
\end{proof}

\begin{proposition}\label{prop:itsastar}
Let ${\cal A}=\{A_1,A_2,\dotsc,A_m\}$ be a bimodal collection of disjoint subsets of an abelian group $G$ with $r_{\!_{\cal A}}\geq 2$.  Then there exists a nonempty subset $D_{\!_{\cal A}}\subseteq G$ such that:
\begin{enumerate}
\item $a_i+H_i=A_i\sqcup D_{\!_{\cal A}}$ for $i\leq r_{\!_{\cal A}}$;
\item  $(a_i+H_i)\cap (a_j+H_j)=D_{\!_{\cal A}}$ for any $i,j\in 1,2,\dotsc, r_{\!_{\cal A}}$ with $i\neq j$;
\item  $D_{\!_{\cal A}}$ is itself a coset of a subgroup of $G$.
\end{enumerate}
\end{proposition}

\begin{proof}
\begin{enumerate}
\item \label{asdf} Let $D_{\!_{\cal A}}=(a_1+H_1)\setminus A_1$.  We will show that $D_{\!_{\cal A}}=(a_i+H_i)\setminus A_i$ for each $i$ with $1\leq i\leq r_{\!_{\cal A}}$.  Let $v\in D_{\!_{\cal A}}$.  We observe that $v\notin A$, by Lemma~\ref{lem:excitinglemma}; in particular, $v\notin A_i$ for any $i\neq 1$.  For $i$ with $2\leq i\leq r_{\!_{\cal A}}$ there exists $h_i\in H_i$ with $h_i\notin H_1$, by Lemma~\ref{lem:notsubgroups}.  As $a_1\in B_i$, we have that $a_1+h_i\in B_i$.  But $a_1+h_i\notin A_1$, so we also have $a_1+h_i+h_1\in B_1$ for any $h_1\in H_1$.  Taking $h_1=v-a_1$, we deduce that $v+h_i\in B_1$.  

We claim that $v+h_i\in A_i$.  For, if it were in $B_i$ this would imply that $v\in B_i$, which contradicts the fact that $v\notin A$.  In turn, this implies $v\in a_i+H_i$. As $v\notin A$ we conclude that $v\in (a_i+H_i)\setminus A_i$.  This shows that $D_{\!_{\cal A}}\subseteq(a_i+H_i)\setminus A_i$; repeating the above argument starting with $v^\prime\in (a_i+H_i)\setminus A_i$ allows us to show that $(a_i+H_i)\setminus A_i\subseteq D_{\!_{\cal A}}$.

\item Follows immediately from \ref{asdf}.
\item The set $D_{\!_{\cal A}}$ is precisely the intersection of all cosets $a_i+H_i$ with $i\leq r_{\!_{\cal A}}$.  This implies that $D_{\!_{\cal A}}$ is a coset of the subgroup obtained by taking the intersection of $H_i$ for all $i\leq r_{\!_{\cal A}}$.
\end{enumerate}
\end{proof}
A collection of sets $F_1,F_2,\dotsc F_k$ with the property that $F_i\cap F_j=D$ for all $i\neq j$ is said to be a {\em $k$-star with kernel $D$} \cite{ErdosRado,furedi}.  Using this terminology, Proposition~\ref{prop:itsastar} shows that for $i$ with $1\leq i \leq r_{\!_{\cal A}}$ the collection of cosets $a_i+H_i$ form an $r_{\!_{\cal A}}$-star with kernel $D_{\!_{\cal A}}$.

By Lemma~\ref{lemma:youcanshift}, a bimodal collection of sets may undergo a shift without affecting its bimodality and so we can choose to shift a collection $\cal A$ with $r_{\!_{\cal A}} \geq 2$ by an element of $D_{\!_{\cal A}}$.  This will allow us to assume, without loss of generality, that $D_{\!_{\cal A}}$ is in fact a subgroup of $G$.

\begin{definition}
A bimodal collection ${\cal A}=\{A_1,A_2,\dotsc,A_m\}$ of disjoint subsets of an abelian group $G$ with $r_{\!_{\cal A}}\geq 2$ will be said to be \emph{in canonical position} if it has been shifted so that $D_{\!_{\cal A}}$ is  a subgroup of $G$.
\end{definition}
\begin{remark}
We observe that if ${\cal A}=\{A_1,A_2,\dotsc,A_m\}$ is a bimodal collection of subsets with $r_{\!_{\cal A}}\geq 2$ in canonical position we have
\begin{itemize}
\item $D_{\!_{\cal A}}$ is the subgroup $\cap_{i=1}^{r_{\!_{\cal A}}} H_i$;
\item for $i=1,2,\dotsc,r_{\!_{\cal A}}$ we have that $A_i=H_i\setminus D_{\!_{\cal A}}$;
\item the subgroups $H_1,H_2,\dotsc, H_{r_{\!_{\cal A}}}$ form an $r_{\!_{\cal A}}$-star with kernel $D_{\!_{\cal A}}$.
 \end{itemize}
\end{remark}
We have thus seen that the sets $A_1$ to $A_{r_{\!_{\cal A}}}$ of a bimodal collection with $r_{\!_{\cal A}}\geq 2$ occur together in a very structured way.  In fact, we will see that these sets also impose considerable structure on the remaining members of $\cal A$.
\begin{proposition}
Let ${\cal A}=\{A_1,A_2,\dotsc,A_m\}$ be a bimodal collection of disjoint subsets of an abelian group $G$ with $r_{\!_{\cal A}}\geq 2$, in canonical position. Then 
\begin{equation*}(H_1+H_2+\dotsb+H_{r_{\!_{\cal A}}}) \setminus D_{\!_{\cal A}}
\end{equation*}
is contained in $A=\cup_{i=1}^m A_i$.
\end{proposition}
\begin{proof}
We begin by showing that $(H_1+H_2)\setminus D_{\!_{\cal A}} \subseteq A$.  Let $h\in (H_1+H_2)\setminus D_{\!_{\cal A}}$.  Then $h=h_1+h_2$ for some $h_1\in H_1$ and $h_2\in H_2$.  If $h_1\in D_{\!_{\cal A}}$ then $h_1+h_2\in H_2$ and hence in $A_2\subseteq A$ as $h_1+h_2\notin D_{\!_{\cal A}}$.  Otherwise, $h_1\in A_1\subseteq B_2$.  This implies that $h_1+h_2$ is also in $B_2\subseteq A$, by Remark~\ref{remark2.4}.

Now consider $h\in (H_1+H_2+H_3)\setminus D_{\!_{\cal A}}$, so $h=h_1+h_2+h_3$ with $h_1\in H_1$, $h_2\in H_2$ and $h_3\in H_3$.  Suppose $h_1+h_2\in D_{\!_{\cal A}}$.  Then $h_1+h_2+h_3\in H_3$ and hence $h_1+h_2+h_3\in A_3$ as $h_1+h_2+h_3\notin D_{\!_{\cal A}}$.  In the case where $h_1+h_2\notin D_{\!_{\cal A}}$, the above argument shows that $h_1+h_2\in A=A_3\cup B_3$.  Now, if $h_1+h_2\in B_3$ then by Remark~\ref{remark2.4} we have that $h_1+h_2+h_3\in B_3\subset A$.  If, however, $h_1+h_2\in A_3$, then $h_1+h_2+h_3\in H_3$ and hence $h_1+h_2+h_3\in A_3$.

Proceeding analogously for all $r_{\!_{\cal A}}$ summands yields the desired result.
\end{proof}
In some circumstances it may be the case that all elements of $(H_1+H_2+\dotsb+ H_{r_{\!_{\cal A}}})\setminus D_{\!_{\cal A}}$ are contained in the union of the sets $A_1,A_2,\dotsc,A_{r_{\!_{\cal A}}}$, but in general this will not be true, and there may be other sets $A_i$ with $i>r_{\!_{\cal A}}$  containing elements of $H_1+H_2+\dotsb +H_{r_{\!_{\cal A}}}$.  Furthermore, there may also be elements of $A$ that do not lie in $H_1+H_2+\dotsb + H_{r_{\!_{\cal A}}}$. The following proposition tells us more about the sets $A_i$ with $i>r_{{\!_{\cal A}}}$.
\begin{proposition}\label{prop:Hi<DA}
Let ${\cal A}=\{A_1,A_2,\dotsc,A_m\}$ be a bimodal collection of disjoint subsets of an abelian group $G$ with $r_{\!_{\cal A}}\geq 2$, in canonical position.  Then for $A_i$ with $i \geq r_{\!_{\cal A}}+1$, the group $H_i$ is a subgroup of $D_{\!_{\cal A}}$.
\end{proposition}
\begin{proof}
Let ${\cal A}=\{A_1,A_2,\dotsc,A_m\}$ be a bimodal collection of disjoint subsets of an abelian group $G$ with $r_{\!_{\cal A}}\geq 2$, in canonical position, and suppose there exists $i> r_{\!_{\cal A}}$ for which $H_i$ is not a subgroup of $D_{\!_{\cal A}}$.  Then there exists $h_i\in H_i$ with $h_i\notin D_{\!_{\cal A}}$.  We observe that there exists $j\leq r_{\!_{\cal A}}$ with $h_i\notin H_j$, as the only elements common to all of the groups $H_t$ with $t\leq r_{\!_{\cal A}}$ are those of $D_{\!_{\cal A}}$. As $a_j\in B_i$ we thus have $a_j+h_i\in B_i$, and furthermore $a_j+h_i\in B_j$.  Let $v\in D_{\!_{\cal A}}$.  Then $v-a_j\in H_j$, and thus $a_j+h_i+(v-a_j)=v+h_i\in B_j$.  

Now, $A=A_i\cup B_i$, and $B_i$ is a union of cosets of $H_i$.  Since $|A_i|=|H_i|$, this implies that in fact $A$ is a union of cosets of $H_i$.  However, we have just shown that $v+h_i\in A$.  As we know that $v\notin A$ (since $v\in D_{\!_{\cal A}}$,) this leads to a contradiction.
\end{proof}

This implies that each set $A_i$ with $i> r_{\!_{\cal A}}$ is a coset of a subgroup of $D_{\!_{\cal A}}$, and hence is contained in a coset of $D_{\!_{\cal A}}$.  Hence each such set is either wholly contained within $H_1+H_2+\dotsb +H_{r_{\!_{\cal A}}}\setminus D_{\!_{\cal A}}$, or else lies completely outside $H_1+H_2+\dotsb + H_{r_{\!_{\cal A}}}$. Thus we now know that the set $H_1+H_2+\dotsb +H_{r_{\!_{\cal A}}}\setminus D_{\!_{\cal A}}$ is entirely partitioned by sets from $\cal A$.   We also know that any $A_i$ occuring outside of this set is a coset of a subgroup of $D_{\!_{\cal A}}$.  The following proposition characterises the union of these $A_i$.

\begin{proposition}
Let ${\cal A}=\{A_1,A_2,\dotsc,A_m\}$ be a bimodal collection of disjoint subsets of an abelian group $G$ with $r_{\!_{\cal A}}\geq 2$, in canonical position. Let $H$ be the group $H=H_1+H_2+\dotsb+H_{r_{\!_{\cal A}}}$.  If $A\setminus H$ is nonempty, then it consists of a union of cosets of $H$; further, the sets $A_i$ which lie in $A \setminus H$ arise from a subdivision of these cosets of $H$.
\end{proposition}
\begin{proof}
We must show that if there is an element $x\in A\setminus H$ then $x+h\in A$ for all $h\in H$.  First suppose $h\in H_i$ for some $i\leq r_{\!_{\cal A}}$.  As $x\in A\setminus H$ we have $x\in B_i$, and so $x+h\in B_i$. Otherwise, $h$ has the form $h_1+h_2+\dotsb+h_{r_{\!_{\cal A}}}$ with $h_1\in H_1,\dotsc,h_{r_{\!_{\cal A}}}\in H_{r_{{\!_{\cal A}}}}$.  But $x\in B_1$ so $(x+h_1)\in B_1\subset A$.  Furthermore, as $x\notin H$ we have $x+h_1\notin H$.  Hence $x+h_1\in B_2\subset A$, from which we deduce $x+h_1+h_2\in B_2$.  Continuing in this manner we conclude that $x+h_1+h_2+\dotsb+h_{r_{\!_{\cal A}}}\in B_{r_{\!_{\cal A}}}\subset A$ as required.  So $A \setminus H$ is a union of cosets of $H$.
Since by Proposition \ref{prop:Hi<DA},  we have that $H_i \leq D_{\!_{\cal A}} \leq H$ for all $i \geq r_{\!_{\cal A}}+1$,  each $A_i$ in $A \setminus H$ must be wholly contained in a single coset of $H$.
\end{proof}


The following theorem fully summarises the structure that has been determined in the above propositions.

\begin{theorem}\label{thm:properties}
Let ${\cal A}=\{A_1,A_2,\dotsc,A_m\}$ be a bimodal collection of disjoint subsets of an abelian group $G$ with $r_{\!_{\cal A}}\geq 2$, in canonical position. Then
\begin{enumerate}
\item The internal difference groups $H_1,H_2,\dotsc,H_{r_{{\!_{\cal A}}}}$ form an $r_{{\!_{\cal A}}}$-star with kernel $D_{{\!_{\cal A}}}$, and for each $i$ with $1\leq i \leq r_{{\!_{\cal A}}}$ we have $A_i=H_i\setminus D_{{\!_{\cal A}}}$.
\item Any set $A_i$ with $i> r_{{\!_{\cal A}}}$ is a coset of a subgroup of $D_{{\!_{\cal A}}}$.
\item If $H$ denotes the group $H_1+H_2+\dotsb+H_{r_{{\!_{\cal A}}}}$, then $H\setminus D_{{\!_{\cal A}}}$ is contained in $A$.  Furthermore, the sets in $\cal A$ can be labelled such that for some $k$ with $r_{{\!_{\cal A}}}\leq k\leq m$ we have that $H\setminus D_{{\!_{\cal A}}}$ is partitioned by $A_1,A_2,\dotsc,A_{k}$.
\item If $k<m$ then the sets $A_i$ with $i>k$ arise from a subdivision of cosets of $H$.
\end{enumerate}
\end{theorem}

We will see that not only are the conditions of Theorem~\ref{thm:properties} necessary for a collection of sets to be bimodal, they are sufficient as well.  We present a technique for constructing a bimodal collection of sets with $r_{\!_{\cal A}}\geq 2$.
\begin{theorem}
Let $G$ be an abelian group, and for $t\geq 2$ let $H_1,H_2,\dotsc,H_t$ be distinct subgroups of $G$ forming a $t$-star with kernel $D$, such that $|H_i:D|>2$ for $i$ with $1\leq i \leq t$.  Let $H=H_1+H_2+\dotsb+H_t$. 

Let $\cal A$ consist of the following subsets of $G$:
\begin{enumerate}
\item all subsets of the form $A_i=H_i\setminus D$ for $i$ with $1\leq i \leq t$;
\item all cosets of $D$ that are subsets of $H$, but are not in $\cup_{i=1}^t H_i$;
\item for any number of cosets of $H$, all the cosets of $D$ that lie within those cosets of $H$.
\end{enumerate} 
Then $\cal A$ is a bimodal collection of subsets of $G$ with $r_{{\!_{\cal A}}}=t$ in canonical position.
\end{theorem}
\begin{proof}
We will prove this result by applying Theorem~\ref{thm:bimodal} to each of the internal difference groups of the sets in $\cal A$. For any subset which is a coset of $D$, its internal difference group is the subgroup $D$ itself.  For any $A_i$ with $1 \leq i \leq t$, its internal difference group is $H_i$: since the index of $D$ in $H_i$ is greater than $2$, we know that $A_i$ contains at least $2$ distinct cosets of $D$.
Fix an element $x \in A_i$ and consider all differences between $x$ and the elements of $A_i$; these will give all elements of $H_i$ except for those in the coset $x+D$.  Now, pick an element $y \in A_i$ such that $x$ and $y$ lie in different cosets of $D$.  Taking all differences between $y$ and the elements of $A_i$ yields all elements of $H_i$ except those in $y+D$.  Hence all elements of $H_i$ lie in the internal difference group of $A_i$, as required.

We observe that every set in $\cal A$ is a union of cosets of $D$.  Thus for any set whose internal difference group is $D$, the union of the remaining sets is a union of cosets of $D$, as required.  For $A_i$ with $1 \leq i \leq t$, we note that $A_i \cup D=H_i$ by construction, so $B_i$ consists of $H \setminus H_i$ together with a union of cosets of $H$, and is therefore a union of cosets of $H_i$.
\end{proof}

\begin{remark}
By Theorem \ref{thm:subdivision}, new bimodal collections can be obtained by subdividing those subsets in the above construction that are cosets of $D$.  Comparing this construction with Theorem~\ref{thm:properties}, we see that all bimodal collections of sets with $r_{\!_{\cal A}}\geq 2$ arise in this way.
\end{remark}


\subsection{The case when $r_{\!_{\cal A}} =1$ }

We now consider the situation where $r_{\!_{\cal A}}=1$, i.e. when precisely one of the sets $A_i$ does not fill the coset of $H_i$ which contains it.

\begin{proposition}
Let ${\cal A}=\{A_1,A_2,\dotsc,A_m\}$ be a bimodal collection of disjoint subsets of an abelian group $G$ with $r_{\!_{\cal A}}=1$.  Then
\begin{enumerate}
\item for all $i$ with $2\leq i \leq m$, $A_i$ is a coset of $H_i$ and $H_i$ is a subgroup of $H_1$; 
\item $A_1$ is a proper subset of a coset of $H_1$, and consists of a union of cosets of the group $D=H_2+H_3+\dotsb+H_m$.
\end{enumerate}
\end{proposition}
\begin{proof}
\begin{enumerate}
\item \label{sadfasdfa}Suppose that for some $i$ with $2\leq i \leq m$ there exists $h_i\in H_i$ with $h_i\notin H_1$.  Let $u\in (a_1+H_1)\setminus A_1$, and let $v$ be in $A_1$.  Then $u-v\in H_1$. We have $v\in B_i$, so $v+h_i\in B_i$.  Since $h_i\notin H_1$ we have that $v+h_i\in B_1$.  This implies that $v+h_i+(u-v)\in B_1$, i.e. that $u+h_i\in B_1\subseteq A$.  Since $A_i$ consists of a coset of $H_i$, we know that $A=A_i\sqcup B_i$ is a union of cosets of $H_i$.  This implies that if $u+h_i\in A$ then $u\in A$.  However, by Lemma~\ref{lem:excitinglemma} we know that no element of $(a_1+H_1)\setminus A_1$ lies in $A$, which gives a contradiction.

\item Let $x\in A_1$.  We want to show that for any $d\in D$ we have $x+d\in A_1$.  Now $d=h_2+h_3+\dotsb+h_m$ where $h_i\in H_i$ for $2\leq i \leq m$.  Since $x\in B_2$ we have $x+h_2\in B_2$.  However, by \ref{sadfasdfa} we know $H_2\leq H_1$, and so $x+h_2\in B_2 \cap (a_1+H_1)=A_1$.  This implies $x+h_2\in B_3$, whence $x+h_2+h_3\in B_3$.  Proceeding in this manner we determine that $x+d\in A_1$ as required.
\end{enumerate}
\end{proof}

\begin{definition}
A bimodal collection ${\cal A}=\{A_1,A_2,\dotsc,A_m\}$ of disjoint subsets of an abelian group $G$ with $r_{\!_{\cal A}}=1$ will be said to be \emph{in canonical position} if it has been shifted so that $A_1 \subseteq H_1$ and the subgroup $D=H_2+H_3+\dotsb+H_m$ is contained in $H_1 \setminus A_1$.
\end{definition}

We summarise the structural properties of the situation.
\begin{theorem}
Let ${\cal A}=\{A_1,A_2,\dotsc,A_m\}$ be a bimodal collection of disjoint subsets of an abelian group $G$ with $r_{\!_{\cal A}}=1$, in canonical position. Let $D=H_2+H_3+\dotsb+H_m$.  Then
\begin{enumerate}
\item $A_1 \subseteq H_1 \setminus D$ and $A_1$ is a union of cosets of $D$;
\item Each set $A_i$ with $2\leq i \leq m$ is a coset of $H_i$, and $A_2,\dotsc,A_m$ arise from a subdivision of cosets of $H_1$.
\end{enumerate}
\end{theorem}

We give a concrete example of a bimodal collection of sets with $r_{\!_{\cal A}}=1$.

\begin{example}Let $G=\mathbb{Z}_{36}$.  The collection $\cal A$ of subsets defined by
$A_1=\{12,15,30,33\},A_2=\{1,19\},A_3=\{4,22\},A_4=\{7,25\},A_5=\{10,28\},A_6=\{13\},A_7=\{16\},A_8=\{31\},A_9=\{34\}$ is bimodal.
We observe that $H_1=\langle3\rangle=\{0,3,6,9,12,15,18,21,24,27,30,33\}$, and $H_2=H_3=H_4=H_5=\langle18\rangle=\{0,18\}$, so $H_2+H_3+H_4+H_5+H_6+H_7+H_8+H_9=\langle18\rangle<H_1$.  Note also that $\cup_{i=2}^9 A_i$ is the coset $1+H_1$, and that $A_1$ is a union of cosets of $\langle 18 \rangle$.  Here we have  $A_1\subset H_1,$ and we note that the set $H_1\setminus A_1$ is not in fact a group.  (This is in contrast to the behaviour of bimodal collections of sets with $r_{\!_{\cal A}}\geq 2$.)
 
\end{example}


We will see that not only are the conditions of the above theorem necessary for a collection of sets to be bimodal, they are sufficient as well.  We present a technique for constructing a bimodal collection of sets.
\begin{theorem}
Suppose we have an abelian group $G$, a subgroup $H_1$ of $G$ and a proper subset $A_1$ of $H_1$ whose internal differences generate $H_1$.
\begin{enumerate}
\item Let $S$ be the set of all subgroups $J \leq G$ with the property that $A_1$ is a union of cosets of $J$.  

\item For any number of cosets of $H_1$, partition these cosets of $H_1$ using only cosets of subgroups from $S$.

\item Take $\cal A$ to consist of $A_1$ together with the sets in the above partition.  
\end{enumerate}
Then ${\cal A}$ forms a bimodal collection of subsets of $G$ with $r_{{\!_{\cal A}}}=1$.
\end{theorem}
\begin{proof}
The bimodality of ${\cal A}$ is immediate, upon applying Theorem~\ref{thm:bimodal} to each of the internal difference groups of the sets in $\cal A$.
\end{proof}

\subsection{The case when $r_{\!_{\cal A}} =0$ }
Finally, we consider the situation when each set $A_i$ in ${\cal A}$ fills its own coset of $H_i$.

\begin{theorem}
Let ${\cal A}=\{A_1,A_2,\dotsc,A_m\}$ be a bimodal collection of disjoint subsets of an abelian group $G$ with $r_{\!_{\cal A}}=0$.  Let $H$ be the group $H=H_1+H_2+\dotsb+H_m$.  Then each $A_i$ ($1 \leq i \leq m$) is a coset of $H_i$, and these sets arise from a subdivision of cosets of $H$. 
\end{theorem}
\begin{proof}
We first prove that $A$ is a union of cosets of $H$.  For each $i$, $A_i$ is a coset of $H_i$ and $B_i$ is a union of cosets of $H_i$, hence $A$ is a union of cosets of $H_i$ for all $i$.
Let $x\in A$ and let $h\in H$.  Then $h$ has the form $h_1+h_2+\dotsb+h_m$ with $h_i\in H_i$ for $i$ with $1\leq i \leq m$.  As $x\in A$ we thus have $x+h_1\in A$.  In turn this implies $x+h_1+h_2\in A$ and so on, so we deduce $x+h\in A$, as required.  Since $H_i \leq H$ for all $1 \leq i \leq m$, every $A_i$ is wholly contained in a single coset of $H$, so $\cal{A}$ is a subdivision of the cosets of $H$.
\end{proof}

\begin{remark}
By applying subdivision to the construction of Lemma~\ref{lem:cosets}, we obtain a bimodal collection of subsets in which each of the sets is a coset of its internal difference group, i.e. a bimodal collection with $r_{\!_{\cal A}}=0$.  We note that the above theorem shows that every such collection arises in this manner.
\end{remark}

\section{Further Questions}\label{sec:conc}
Now that we understand the full range of possible structures of bimodal collections of sets in abelian groups, we are led to some natural further questions.  The bimodal property arose in the context of studying external difference families corresponding to R-optimal AMD codes; it would be interesting to explore the potential for wider applications of this concept in related areas.   Group partitions, and particularly vector space partitions, also have a range of applications, and it would be of interest to investigate the role played by the bimodality of these structures in these applications.

Group partitions are known to exist for certain classes of nonabelian group, such as Frobenius groups, and these have been shown to give rise to bimodal collections of sets \cite{WEDF}.  It would be natural to seek to develop a comparable understanding of bimodality in a nonabelian context, although care must be taken to refine the relevant definitions where appropriate.

  
\bibliography{edfbib}

\end{document}